\documentclass[smallcondensed,12pt,envcountsect]{svjour3}
\usepackage{graphicx}
\usepackage{amsfonts}
\usepackage{amssymb}
\usepackage[centertags]{amsmath}
\usepackage{newlfont}

\usepackage{color} 
\usepackage{graphicx,color} 
\usepackage{amsmath, amssymb, graphics}

\usepackage{graphicx}
\def\r{\mathbb R}
\def\s{\mathbb S}

\begin{document}

\title{Compact $\lambda$-translating solitons with boundary\thanks{Partially
supported by MEC-FEDER
 grant no. MTM2014-52368-P}
}
  
\author{Rafael L\'opez}

\institute{Rafael L\'opez \at
              Departamento de Geometr\'{\i}a y Topolog\'{\i}a\\ Instituto de Matem\'aticas (IEMath-GR)\\
 Universidad de Granada\\
 18071 Granada, Spain \\ 
  \email{rcamino@ugr.es}}

\date{Received: date / Accepted: date}

\maketitle
\begin{abstract}
A $\lambda$-translating soliton  with density vector $\vec{v}$   is a surface $\Sigma$ in Euclidean space $\r^3$ whose mean curvature $H$ satisfies $2H=2\lambda+\langle N,\vec{v}\rangle$, where $N$ is the Gauss map of $\Sigma$. In this article we study the shape of a  compact $\lambda$-translating soliton in terms of its boundary. If $\Gamma$ is a given closed curve,  we deduce under what conditions on $\lambda$ there exists a compact $\lambda$-translating soliton $\Sigma$ with boundary $\Gamma$ and we provide estimates of the surface area in relation with the height of $\Sigma$. Finally we study the shape of $\Sigma$ related with the one of $\Gamma$, in particular, we give conditions that assert that $\Sigma$ inherits the symmetries of its boundary $\Gamma$.
\end{abstract}

\section{Introduction}

Let us fix a unit vector $\vec{v}$ in Euclidean space $\r^3$ and $\lambda$ a real number. In this paper we study oriented surfaces $\Sigma$ whose mean curvature $H$ satisfies the equation
\begin{equation}\label{eq1}
H(p)=\lambda+\frac{\langle N(p),\vec{v}\rangle}{2},\quad (p\in\Sigma),
\end{equation}
where $N$ is the Gauss map of $\Sigma$. The interest of this equation is due to its relation with manifolds with density. Indeed, consider $\r^3$ with a  positive  density function $e^\phi$, $\phi\in C^\infty(\Sigma)$, which serves as a weight for  the volume and the surface area. The first variation   of the area $A_\phi(t)$ with density $e^\phi$ under compactly supported variations of $\Sigma$ and with variation vector field $\xi$ is 
$$\frac{d}{dt}{\Big|}_{t=0}A_\phi(t)=-2\int_\Sigma H_\phi\langle N,\xi\rangle  dA_\phi,$$
where $H_\phi=H-\frac12\frac{d\phi}{dN}$. Then it is immediate that   $\Sigma$ is a critical point of the  area $A_\phi$ for a  prescribed weighted volume if and only if  $H_\phi$ is a constant function $H_\phi=\lambda$: see \cite{gr,mo}. If we now take  $\phi:\r^3\rightarrow\r$  the height function $\phi(q)=\langle q,\vec{v}\rangle$, then the expression $H-(d\phi/dN)/2=\lambda$ is just Eq. (\ref{eq1}). An interesting case of (\ref{eq1}) is when $\lambda=0$, 
because the equation $H=\langle N,\vec{v}\rangle/2$ appears in the singularity theory of the mean curvature flow, indeed, it is  the  equation of the limit flow by a proper blow-up procedure near type II singular points (\cite{hs,il}). In such a case the surface is called   a \emph{translating soliton} of the mean curvature flow, or simply a translator \cite{wh}.   Recently, translating solitons have been widely studied  and we refer to the reader the next bibliography without to be a complete list: \cite{aw,css,ha,mar,mh,pyo,sh,sm,wa}.

Motivated by the case $\lambda=0$, we say that $\Sigma$  is a {\it $\lambda$-translating soliton} if (\ref{eq1}) holds everywhere on $\Sigma$ and the vector $\vec{v}$   is called  the \emph{density vector}.  Equation \eqref{eq1} can viewed as a type of prescribing mean curvature equation, such as it occurs for the constant mean curvature (cmc in short) equation. In fact,   the equation $H_\phi=0$ in a nonparametric form  appeared in the classical article of Serrin  (\cite[p. 477--478]{se}) and was studied  in the context of the maximum principle of elliptic equations.

Examples of $\lambda$-translating solitons are planes parallel to $\vec{v}$ ($\lambda=0$), planes orthogonal to $\vec{v}$ ($\lambda=1/2$) and right circular  cylinders of radius $r>0$ with axis parallel to $\vec{v}$ ($\lambda = r/2$). Recently,  the author has classified all $\lambda$-translating solitons that are invariant by a group of translations (cylindrical surfaces) and a group of rotations (surfaces of revolution): see   \cite{lo2}.

 In this paper we study the shape of a compact $\lambda$-translating soliton in terms of the geometry of its boundary. In order to fix the terminology, let $\Gamma\subset\r^3$ be a closed curve and let $\psi:\Sigma\rightarrow\r^3$ be an immersion of  a connected oriented surface $\Sigma$ with smooth boundary $\partial \Sigma$. We say that $\Gamma$ is the boundary of $\psi$ (or simply the boundary of $\Sigma$ if $\psi$ is understood),  if  the immersion $\psi$ restricted to the boundary $\partial \Sigma$ is a diffeomorphism onto $\Gamma$. In case that $\psi$ is an embedding, we say that $\partial\Sigma$ is the boundary of $\psi$. 
 
 Using the divergence theorem, we deduce in Th. \ref{t1} that there do not exist closed $\lambda$-translating solitons. In particular, if $\Sigma$ is a compact surface, then its boundary $\partial\Sigma$ is a non-empty set. If $\Gamma\subset\r^3$ is a given closed curve,   we ask   whether there exist necessary conditions on $\lambda$ for the existence of  a compact $\lambda$-translating soliton with boundary $\Gamma$. We   prove in Th. \ref{tlambda}  that not all values $\lambda$ are possible and that $\lambda$ must have a certain relation with the geometry of $\Gamma$. 
 
 A second question that we address is how the geometry of the boundary $\Gamma$ affects on the shape of the $\lambda$-translating soliton that spans. In Sec. \ref{t-b}, we assume that $\Gamma$ lies in a plane $\Pi$ of $\r^3$. Then we study when $\Sigma$ lies in one side of $\Pi$. Here we will use the maximum principle for Eq. (\ref{eq1}) that allows to establish  comparison arguments between two $\lambda$-translating solitons that touch at some point. In Th. \ref{tarea} we give an estimate of the area of a $\lambda$-translating soliton in terms of its height about $\Pi$. Finally in Sec. \ref{tsym} we study if a compact $\lambda$-translating soliton $\Sigma$ inherits the symmetries of its boundary. To be precise, consider $\Gamma\subset\r^3$ a given closed curve such that $\Gamma$ is invariant by a rigid motion $M:\r^3\rightarrow\r^3$.  If $\Sigma$ is a compact $\lambda$-translating soliton with boundary $\Gamma$, we ask whether $\Sigma$ is also invariant by $M$, that is, if $M(\Sigma)=\Sigma$. The simplest case of boundary  is when $\Gamma$ is a round circle   and we ask if $\Sigma$ is a surface of revolution. It is proved in  \cite{pyo} (see also \cite{pe}) that a  compact translating soliton ($\lambda=0$) spanning a circle contained in a plane orthogonal to $\vec{v}$ is a surface of revolution.  
In this paper we extend  this result by proving in Cor. \ref{c2} that rotational surfaces are the only embedded compact $\lambda$-translating solitons with circular boundary   that  lie in one side of the boundary plane.

\section{Preliminaries and first results}

In Euclidean space $\r^3$ we stand for $(x,y,z)$ the canonical coordinates. We will use the terminology   horizontal (resp. vertical) to be orthogonal (resp. parallel)  to the $z$-direction. It is immediate that if we reverse the orientation on a $\lambda$-translating soliton, then we get a $-\lambda$-translating soliton. It is also clear that any rigid motion of $\r^3$ that leaves invariant the term $\langle N(p),\vec{v}\rangle$ in (\ref{eq1}) is a transformation that preserves the value of $H_\phi$. In particular, this occurs for a translation,  a rotation about a straight-line parallel to $\vec{v}$ and a reflection about a plane parallel to $\vec{v}$.

 If we ask for those closed surfaces (compact without boundary)   that satisfy (\ref{eq1}), we have:

\begin{theorem} \label{t1}
There are no closed $\lambda$-translating solitons.
\end{theorem}

\begin{proof} By contradiction, let $\psi:\Sigma\rightarrow\r^{3}$ be an immersion of a closed surface $\Sigma$ whose mean curvature $H$ satisfies (\ref{eq1}). It is known that if $\vec{a}\in\r^{3}$,   the Laplace-Beltrami operator $\Delta$ on $\Sigma$ of the height function $\langle \psi,\vec{a}\rangle$ is $\Delta\langle\psi,\vec{a}\rangle=2H\langle N,\vec{a}\rangle$. If we now put  $\vec{a}=\vec{v}$, we have  
\begin{equation}\label{lapla}
\Delta\langle\psi,\vec{v}\rangle=2\lambda\langle N,\vec{v}\rangle +\langle N,\vec{v}\rangle^2.
\end{equation}
 Integrating this identity in $\Sigma$, using the divergence theorem and because $\partial\Sigma=\emptyset$, we have
\begin{equation}\label{com}
0=2\lambda\int_\Sigma \langle N,\vec{v}\rangle\ d\Sigma+\int_\Sigma\langle N,\vec{v}\rangle^2\ d\Sigma.
\end{equation}
Since the constant vector field in $\r^3$ defined by $Y(p)=\vec{v}$ has zero divergence, and $\Sigma$ is a closed surface, the divergence theorem gives  $\int_\Sigma\langle N,\vec{v}\rangle d\Sigma=0$. We conclude from \eqref{com} that $0=\int_\Sigma\langle N,\vec{v}\rangle^2\ d\Sigma$, that is, $\Sigma$ is included in a plane parallel to $\vec{v}$, a contradiction.
\end{proof}

As a consequence of Th. \ref{t1}, any compact $\lambda$-translating soliton $\Sigma$ has non-empty boundary $\partial\Sigma$. This contrasts with the theory of cmc    surfaces, where there are many examples closed surfaces. 

Equation (\ref{lapla}) allows to answer  the question if for a given closed curve $\Gamma$   there exist necessary conditions on the value  $\lambda$ for the existence of a $\lambda$-translating soliton with boundary $\Gamma$.  Let $\Sigma$ be a compact $\lambda$-translating soliton oriented by $N$. Let $K$ be a compact  oriented surface with $\partial K=\partial\Sigma$ and such that $\Sigma\cup K$ is an oriented $2$-cycle of $\r^3$. Let $\eta_K$ be the unit normal vector field induced on $K$. The divergence theorem gives 
\begin{equation}\label{nn}
\int_\Sigma \langle N,\vec{v}\rangle\ d\Sigma+\int_{K}\langle\eta_K,\vec{v}\rangle\ dK=0.
\end{equation}
An integration of (\ref{lapla}) in $\Sigma$ together (\ref{nn}) gives  
\begin{eqnarray*}
-\int_{\partial\Sigma}\langle\nu,\vec{v}\rangle ds&=&2\lambda\int_\Sigma\langle N,\vec{v}\rangle\ d\Sigma+\int_\Sigma\langle N,\vec{v}\rangle^2\ d\Sigma\geq 2\lambda\int_\Sigma\langle N,\vec{v}\rangle\ d\Sigma\\
&=&-2\lambda \int_K\langle \eta_K,\vec{v}\rangle\ dK,
\end{eqnarray*}
where $\nu$ is the unit inward conormal vector along $\partial\Sigma$. Hence 
$$2|\lambda|\left|\int_K \langle \eta_K,\vec{v}\rangle\ dK\right|\leq L,$$
where $L$ is the length of $\Gamma$. As a conclusion, we have:

\begin{theorem}\label{tlambda} Let $\Gamma$ be  a Jordan curve. If $\Sigma$ is a compact $\lambda$-translating soliton with boundary $\Gamma$, then 
\begin{equation}\label{nec}
|\lambda|\leq \frac{L}{2\left|\int_K\langle\eta_K,\vec{v}\rangle dK\right|},
\end{equation}
for any orientable compact surface $K$ with $\partial K=\Gamma$ and $\int_K\langle \eta_K,\vec{v}\rangle dK\not=0$. In the particular case that $\Gamma$ is included in a plane $\Pi$ which is not parallel to $\vec{v}$, we have 
$$|\lambda|\leq \frac{L}{2|\langle\vec{v},\vec{a}\rangle|\mbox{ area}(D)}$$
where $D\subset P$ is the  domain bounded by $\Gamma$ and $\vec{a}$ is a unit vector orthogonal to $\Pi$. 
\end{theorem}
We observe that the right-hand side in (\ref{nec}) does not depend on $\Sigma$ and thus, if $\lambda\not=0$, Th. \ref{tlambda} implies that not all values of $\lambda$ are possible, but there exists an upper bound that depends only on the geometry of $\Gamma$. 

 We extend the above argument in case that $\partial\Sigma$ has more than one component.  Let $\psi:\Sigma\rightarrow\r^3$ be an immersion of a compact oriented surface $\Sigma$ where $N:\Sigma\rightarrow\s^2$ is its orientation.  Let $\partial\Sigma=C_1\cup\ldots\cup C_m$, where $C_i$ are topological circles. We attach topological disks $\Omega_i$ to $\Sigma $ along $C_i$ and we obtain an orientable two dimensional compact connected topological surface 
$$\tilde{\Sigma}=\Sigma\cup\left(\bigcup_{i=1}^m\Omega_i\right)$$
without boundary. Let $\Omega=\cup_{i=1}^m\Omega_i$ and let $\eta_\Omega$ be the induced orientation on $\Omega$. We assume that the restriction of the immersion $\psi$ to $C_i$, $\psi_{|C_i}:C_i\rightarrow\r^3$, is an embedding and let $\psi(C_i)=\Gamma_i$. Consider $K_i$ a compact surface spanning $\Gamma_i$. We can extend $\psi$ continuously to $\tilde{\psi}:\tilde{\Sigma}\rightarrow\r^3$ such that each $\tilde{\psi}_{|\Omega_i}$ is a diffeomorphism of $\Omega_i$ onto $K_i$. The   argument before Th. \ref{tlambda} yields now
\begin{equation}\label{n3}
-\int_{\partial\Sigma}\langle\nu,\vec{v}\rangle ds\geq-2\lambda\sum_{i=1}^m\int_{C_i}\langle\eta_{\Omega},\vec{v}\rangle dK.
\end{equation}
An interesting case appears when each component $\Gamma_i$ lies in a plane $\Pi_i$ which is not parallel to $\vec{v}$. We take $D_i\subset\Pi_i$ the   domain bounded by $\Gamma_i$. Then $\eta_{\Omega}$ is a constant vector $\vec{a}_i$ or $-\vec{a}_i$, where $\vec{a}_i$ is a unit vector orthogonal to $\Pi_i$ chosen to have $\langle\vec{a}_i,\vec{v}\rangle>0$. For $1\leq i\leq m$, set 
$$\mbox{sgn}(N,i)=\left\{\begin{array}{ll} 
+1& \mbox{if $\eta_{\Omega}=-\vec{a}_i$ on $\Omega_i$,}\\
-1& \mbox{if $\eta_{\Omega}=\vec{a}_i$ on $\Omega_i$.}
\end{array}\right.$$
With the above notation and from (\ref{n3}), we conclude:

\begin{corollary} Let $\psi:\Sigma\rightarrow\r^3$ be a $\lambda$-translating soliton where $\Sigma$ is a compact surface whose boundary $\psi(\partial\Sigma)$ is a finite number of Jordan curves $\Gamma_1,\ldots,\Gamma_m$ included each one in a plane $\Pi_i$. If for all $1\leq i\leq m$,   $\Pi_i$ is not parallel to the density vector $\vec{v}$, then 
$$2\lambda\sum_{i=1}^m\mbox{sgn}(N,i)\langle\vec{a}_i,\vec{v}\rangle\mbox{area}(D_i)\leq\sum_{i=1}^m L_i,$$
where $L_i$ is the length of $\Gamma_i$.
\end{corollary}

\section{Compact $\lambda$-translating solitons with planar boundary}\label{t-b}

Theorem \ref{t1} is known for translating solitons and the proof invokes the tangency principle, a geometric version of the maximum principle for Eq. (\ref{eq1}). Indeed, we write Eq. (\ref{eq1}) in  terms of a local parametrization of the surface. After a change of coordinates, we suppose that the density vector is $\vec{v}=(0,0,1)$.  Then a surface   in nonparametric form $z=u(x,y)$   is a $\lambda$-translating soliton if $u$ satisfies
\begin{equation}\label{mean1}
\mbox{div}\left(\frac{Du}{\sqrt{1+|Du|^2}}\right)=2\lambda+\frac{1}{\sqrt{1+|Du|^2}}.
\end{equation}
 Equation (\ref{mean1})  is of elliptic type and it satisfies a maximum principle (\cite{gt,se}), which    can be formulated as follows:  

\begin{proposition}[Tangency principle] Let $\Sigma_1$ and $\Sigma_2$ be  two surfaces with $H_\phi^i$-mean curvatures,  respectively. Suppose $\Sigma_1$ and $\Sigma_2$ are tangent at a common interior point $p$ and the corresponding Gauss maps $N_1$ and $N_2$ coincide at $p$. If $H_\phi^2\leq H_\phi^1$ in a neighbourhood of $p$, then it is not true that $\Sigma_2$ lies above $\Sigma_1$ near $p$ with respect to $N_1(p)$, unless $\Sigma_1=\Sigma_2$ in a neighbourhood of $p$. If $p\in\partial \Sigma_1\cap \partial \Sigma_2$ is a boundary point, the same holds if, in addition,  we have $T_p\partial \Sigma_1=T_p\partial \Sigma_2$. 
\end{proposition}

A consequence of the expression of $H_\phi$ in (\ref{mean1}) is that a  surface with constant $H_\phi=\lambda$ is real analytic and consequently, if two     surfaces $\Sigma_1$ and $\Sigma_2$ with the same  constant $H_\phi$ coincide in an open set, then $\Sigma_1$ and $\Sigma_2$ coincide everywhere. In  case that $\Sigma_1$ and $\Sigma_2$ are translating solitons,  the condition on the orientations at the common point can be dropped because $H_\phi=0$ holds for any orientation.

The proof of Th. \ref{t1} for translating solitons ($\lambda=0$) is as follows. Because $\Sigma$ is a closed surface, let $P$ be a plane parallel to $\vec{v}$ which is tangent to $\Sigma$ at some point $p\in\Sigma$ and $\Sigma$ lies in one side of $P$.  Since $P$ is a translating soliton, the tangency principle implies that $\Sigma$ is contained in $P$ by analyticity, which it is not possible.   We observe that this argument fails when $\lambda\not=0$.

In this section we study  compact $\lambda$-translating solitons whose boundary is included in a plane $\Pi$.   Firstly,  we consider   translating solitons,  proving that the surface    is a domain of the very plane $\Pi$ or it  is contained in one of the two halfspaces determined by the plane $\Pi$.  

\begin{theorem}\label{t2}  Let $\Sigma$ be a translating soliton   whose boundary $\Gamma$ is contained in a plane $\Pi$. Then we have two possibilities:
\begin{enumerate}
\item $\Pi$ is parallel to $\vec{v}$ and the surface $\Sigma$ is included in $\Pi$.
\item  $\Pi$ is not parallel to $\vec{v}$, the interior of $\Sigma$, $\mbox{int}(\Sigma)$, lies  in one side of $\Pi$, and  $\Sigma$ is not tangent  to $\Pi$ at any boundary point. 
\end{enumerate}
\end{theorem}

\begin{proof} After a change of coordinates, we suppose that   $\Pi$ is the horizontal plane of equation $z=0$. Let  $\eta$ be a unitary orientation on $\Pi$ which leads that   $\Pi$ is a $\lambda$-translating soliton for $\lambda=-\langle\eta,\vec{v}\rangle/2$. We have two possibilities.
\begin{enumerate}
\item    $\Pi$ is   parallel   to $\vec{v}$. If $\Sigma$ is not included in $\Pi$, consider $P$ a plane parallel to $\Pi$ in the side of $\Pi$ where $\Sigma$ has points and move $P$ far from $\Sigma$ so $P\cap\Sigma=\emptyset$. Translate $P$ towards $\Sigma$ parallel to $\Pi$ until the first touching point. Since this point is an interior point of $\Sigma$, and $P$ and $\Sigma$ are translating solitons,   the tangency principle and analyticity  implies   $\Sigma\subset P$, a contradiction because $\partial\Sigma\subset\Pi$ and $\Pi\not=P$.  
\item   $\Pi$ is not parallel to $\vec{v}$. By reversing the orientation $\eta$  on $\Pi$ if necessary, we suppose $\lambda>0$. We prove that the interior of $\Sigma$ does not contain points in the side of $\Pi$ where $-\eta$ points to. By contradiction, let $p\in\mbox{int}(\Sigma)$ be the lowest point of $\Sigma$  with respect to the direction $\eta$. If $P$ is the affine tangent plane to $\Sigma$ at $p$, then $P$ is a $\lambda$-translating soliton for the orientation $\eta$. Let $N$ be the orientation on $\Sigma$ such that $N(p)=\eta$. We observe again that $\Sigma$ is a translating soliton for any orientation. With this choice of $N$, the surface $\Sigma$ lies above $P$ in a neighbourhood of $p$. Since $\lambda\not=0$, the tangency principle implies that $0>\lambda$: a contradiction. As a conclusion, $\Sigma$ lies above $\Pi$. Finally, the same argument with the tangency principle proves that  $\mbox{int}(\Sigma)\cap \Pi=\emptyset$ (interior version) and $\Sigma$ is not tangent  to $\Pi$ at any boundary point (boundary version). 
\end{enumerate}
\end{proof}

There are two keys in the above proof. Firstly, the given surface $\Sigma$ is a translating soliton and thus  if we reverse its orientation, the property to be a translating soliton is preserved. The second fact that we utilize is that the ambient space is foliated by a uniparametric   of $\lambda$-translating solitons, namely, all planes parallel to $\Pi$.  

 A consequence of technique employed in the proof of item (1) of Th. \ref{t2} is that when we move a plane parallel to $\vec{v}$ towards a translating soliton, the first touching point must be a boundary point and hence we conclude: 
\begin{corollary}\label{prv} Let $\Sigma$ be a compact translating soliton whose boundary $\Gamma$ is included in a plane $\Pi$ which is not parallel to $\vec{v}$. Then the interior of $\Sigma$ is included in the solid cylinder $\Omega\times\r\vec{v}$, where $\Omega\subset Sp\{\vec{v}\}^\bot$ is the   domain bounded by the convex hull of  $\pi(\Gamma)$, and $\pi:\r^3\rightarrow Sp\{\vec{v}\}^\bot$ is the orthogonal projection on the  plane $Sp\{\vec{v}\}^\bot$. 
\end{corollary}

Theorem  \ref{t2} does not hold for a $\lambda$-translating soliton when $\lambda\not=0$ because we need to know which is the orientation on $\Sigma$ to apply the tangency principle. However, and motivated by the main result in \cite{ko} in the context of cmc surfaces, we have:  

\begin{theorem}\label{tko} Let $\Pi$ be a plane and let $\Gamma\subset\Pi$ be a Jordan curve. Denote by $D\subset\Pi$ the   domain bounded by $\Gamma$ and let $\mbox{ext}(D)=\Pi\setminus\overline{D}$. If $\Sigma$ is an embedded compact $\lambda$-translating soliton with boundary $\Gamma$ and $\Sigma\cap \mbox{ext}(D)=\emptyset$, then the interior of $\Sigma$ lies in one side of $\Pi$ or $\Sigma$ is contained in the plane $\Pi$.
\end{theorem}

\begin{proof}  After a change of coordinates, we suppose that $\Pi$ is the plane of equation $z=0$. Let $\s^1(r)\subset\Pi$ be a circle of radius $r>0$ and denote $\s^2_{-}(r)$ the lower halfsphere contained in the halfspace $z\leq 0$ with $\partial\s^2_{-}(r)=\s^1(r)$. Denote $\Omega(r)\subset \Pi$ the domain bounded by $\s^1(r)$. Since $\Sigma$ is a compact surface, let $r>0$ be sufficiently big such that $\Sigma\cap\{z\leq 0\}$ is contained in the domain of $\r^3$ bounded by $\Omega(r)\cup\s^2_{-}(r)$. Then 
$$T=\Sigma\cup(\Omega(r)\setminus D)\cup\s^2_{-}(r)$$
is an embedded closed surface of $\r^3$, possibly not smooth along $\Gamma\cup\s^1(r)$, that separates $\r^3$ in two connected componentes. Let $W\subset\r^3$ be the bounded component  and we take on $T$ the orientation pointing towards outside $W$. Let $N$ be the induced orientation on $\Sigma$. 

The proof of Th. \ref{tko}  is by contradiction. Suppose that the interior of $\Sigma$ has points in both sides of $\Pi$. Let $p,q\in\Sigma$ be the points of minimum and maximum height about $\Pi$. In particular, $z(p)<0<z(q)$.   As $N$ points outside $W$, then $N(p)=N(q)=e_3$, where $e_3=(0,0,1)$. We consider the affine tangent planes $T_p\Sigma$ and $T_q\Sigma$ oriented by $e_3$ and thus both planes are $\mu$-translating solitons with $\mu=-\langle e_3,\vec{v}\rangle/2$. Using the tangency principle, and since $T_p\Sigma$ lies below $\Sigma$ around $p$, we have $\mu<\lambda$. Similarly, comparing $T_q\Sigma$ and $\Sigma$ at $q$, we have $\lambda<\mu$. This gives a contradiction. 

As a conclusion, $\Sigma$ lies in one side of $\Pi$. Without loss of generality, we suppose $\lambda\leq\mu$. The above argument proves that $\Sigma$ lies in the halfspace $z\geq 0$. We have two possibilities. If the interior of $\Sigma$ lies in one side of $\Pi$,  the result is proved. On the contrary,  there exists $p\in\Sigma\cap D$, and thus $p$ is an interior point, $N(p)=e_3$ and the tangency principle says $\mu\leq\lambda$. This implies $\lambda=\mu$ and $\Sigma\subset\Pi$ by analiticity. \end{proof}

\begin{remark}\label{rko}  Theorem \ref{tko} holds for a more general case when $\Gamma$ is formed by a finite number of disjoint Jordan curves, namely, $\Gamma=\bigcup_{i=1}^m\Gamma_i$, all them contained in a plane $\Pi$. Following \cite[Lem 1]{ko}, the outside of $\Gamma$ is $G=\bigcap_{i=1}^m (\Pi\setminus \overline{D_i})$, where $D_i\subset\Pi$ is the   domain bounded by $\Gamma_i$. Then Th. \ref{tko} holds if we replace $\Sigma\cap \mbox{ext}(D)=\emptyset$ by 
$\Sigma\cap G=\emptyset$.
\end{remark}

  \begin{remark}\label{rside} We can say in what side is $\Sigma$. Let $\eta$ be a fix   orientation on $\Pi$ and let $\mu=-\langle\eta,\vec{v}\rangle/2$. If $\lambda<\mu$, then $\Sigma$ lies in the halfspace determined by $\Pi$ where $\eta$ points.  \end{remark}

 
As a particular case of Th. \ref{tko} and Rem. \ref{rko}, we have:

\begin{corollary}\label{cside} Let $\Pi\subset\r^3$ be a plane. If $\Sigma$ is a compact $\lambda$-translating soliton which is a graph on $\Pi$ and $\partial\Sigma\subset\Pi$, then $\mbox{int}(\Sigma)$ lies in one side of $\Pi$ or $\Sigma$ is a subset of $\Pi$.
\end{corollary}

 \section{An area estimate for graphs}
 
 In this section we give an estimate of the area of a compact $\lambda$-translating graph whose boundary lies in a plane orthogonal to the density vector. Estimates for the area for translating solitons were obtained in \cite{sh} for  the intersection of a convex translating graph with Euclidean balls. Our result consider the intersection of a $\lambda$-translating graph with halfspaces of $\r^3$.

\begin{theorem}\label{tarea} Let $\Sigma$ be a compact $\lambda$-translating soliton    whose boundary is contained in a plane $\Pi$ orthogonal to the density vector $\vec{v}$. Suppose $\Sigma$ is a graph on $\Pi$ and let us orient $\Sigma$ by the unit normal vector $N$ such that  $\langle N,\vec{v}\rangle>0$. Then
\begin{equation}\label{ine}
4\pi h\leq  |1+2\lambda|\ A,
\end{equation}
where  $h$ denotes the height of $\Sigma$ with respect to $\Pi$ and $A$ is the area of $\Sigma$. In particular, if $\lambda\not=-1/2$, then 
$$ \frac{4\pi}{|1+2h|}h\leq A.$$
\end{theorem}

\begin{proof}  After a change of coordinates, we suppose that $\vec{v}=e_3=(0,0,1)$ and   $\Pi$ is the plane of equation $z=0$. By Cor. \ref{cside} we know that the interior of $\Sigma$ lies in one side of $\Pi$ if $\lambda\not=-1/2$ or $\Sigma$ is included in $\Pi$ if $\lambda=-1/2$. In the latter case, the inequality (\ref{ine}) holds trivially. 

Without loss of generality,   we assume $\lambda>-1/2$: a similar the argument holds if   $\lambda<-1/2$. Then Rem. \ref{rside} asserts that the interior of $\Sigma$ lies in the halfspace $z<0$ and thus  the height $h$ of $\Sigma$  is $h=-\min\{z(p):p\in \Sigma\}$.  Consider the height function $g:\Sigma\rightarrow\r$, $g(p)=z(p)=\langle p,e_3\rangle$. For each $t<0$, let $A(t)$ be the area of $\Sigma(t)=\{p\in \Sigma:g(p)\leq t\}$ and let $\Gamma(t)=\{p\in \Sigma:g(p)=t\}$. By the coarea formula (\cite[Th. 5.8]{sa}) we have
$$A'(t)=\int_{\Gamma(t)}\frac{1}{|\nabla g|} ds_t,\quad (t\in \mathcal{R}),$$
 where  $ds_t$ is the line element of $\Gamma(t)$ and $\mathcal{R}$ is the set of all regular values of $g$. If $L(t)$ denotes the length of the planar curve $\Gamma(t)$, then the Schwarz inequality yields
\begin{equation}\label{l2}
L(t)^2\leq\int_{\Gamma(t)}|\nabla g| ds_t \int_{\Gamma(t)}\frac{1}{|\nabla g|}ds_t=A'(t)\int_{\Gamma(t)}|\nabla g| ds_t,\quad (t\in\mathcal{R}).
\end{equation}
Since $\Gamma(t)$ is a level  curve of $g$, we have
$$|\nabla g|^2=\langle\nu^t,e_3\rangle^2,$$
where $\nu^t$ is the unit outward conormal vector of $\Sigma(t)$ along $\Gamma(t)$. As $\Sigma(t)$ lies below the plane $\Pi(t)=\{p\in\r^3: g(p)=t\}$, then 
$$|\nabla g|=\langle\nu^t,e_3\rangle$$
along $\Gamma(t)$.  Now \eqref{l2} writes as
\begin{equation}\label{l3}
L(t)^2\leq A'(t)\int_{\Gamma(t)}\langle\nu^t,e_3\rangle ds_t.
\end{equation}
The curve $\Gamma(t)$, possible non-connected, bounds a union of finitely compact connected domains in $\Pi(t)$, namely,  $\Omega(t)=\Omega_1(t)\cup\ldots \cup\Omega_{n_t}(t)$. In order to estimate the right-hand side of (\ref{l3}), we use the divergence theorem in Eq.  (\ref{lapla}), and   replace  $K$ by  $\Omega(t)$ in Eq. (\ref{nn}). Then we obtain
\begin{eqnarray*}
\int_{\Gamma(t)}\langle\nu^t,e_3\rangle ds_t &=&2\lambda\int_{\Sigma(t)} \langle
N,e_3\rangle\ d\Sigma(t)+\int_{\Sigma(t)}\langle N,e_3\rangle^2 d\Sigma(t)\\
&\leq&2\lambda\int_{\Sigma(t)} \langle
N,e_3\rangle\ d\Sigma(t)+\int_{\Sigma(t)}\langle N,e_3\rangle d\Sigma(t)\\
&=&(1+2\lambda) \int_{\Sigma(t)}\langle N,e_3\rangle d\Sigma(t)\\
&=&(1+2\lambda) \sum_{i=1}^{n_t}\mbox{area}(\Omega_i(t)) =(1+2 \lambda)\mbox{ area}(\Omega(t)).
\end{eqnarray*}
In the first inequality we have used that   $\langle N,e_3\rangle^2\leq\langle N,e_3\rangle$   holds in the graph $\Sigma$.  By substituting into \eqref{l3}, we have 
\begin{equation}\label{l4}
L(t)^2\leq (1+2\lambda) \mbox{ area}(\Omega(t)) A'(t). 
\end{equation}
 If $L_i(t)$ is the length of the boundary of $\Omega_i(t)$, the classical  isoperimetric inequality yields
 $$L(t)^2\geq\sum_{i=1}^{n_t}L_i(t)^2\geq 4\pi\sum_{i=1}^{n_t}\mbox{area}(\Omega_i(t)) =4\pi \ \mbox{area}(\Omega(t)).$$
This inequality and  (\ref{l4}) gives
$$4\pi\leq (1+2\lambda) A'(t),\quad (t\in\mathcal{R}.)$$
By integrating this inequality from $t=-h$ to $t=0$, we conclude
$$4\pi h\leq (1+2\lambda) (A(0)-A(-h))=(1+2\lambda)A.$$

\end{proof}
We point out that the statement of Th. \ref{tarea} can be formulated if we choose the orientation on the graph so $\langle N,\vec{v}\rangle<0$, by replacing the term $|1+2\lambda|$ by $|1-2\lambda|$.  Finally, in the particular case $\lambda=0$, the   estimate (\ref{ine}) establishes that the area $A$ of a translating soliton  with planar boundary contained in a plane orthogonal to $\vec{v}$ satisfies  $4\pi h<A$.

\section{Symmetries of a compact $\lambda$-translating soliton}\label{tsym}
In this section we give some results answering whether a compact $\lambda$-translating soliton inherits the symmetries of its boundary.  In this context, we use the  Alexandrov reflection method \cite{al} which, by means of reflections about a uniparametric family of planes  and  the tangency principle, allows to compare the given surface with itself.

\begin{theorem}\label{t3} Let $\Gamma$ be a Jordan curve contained in a plane $\Pi$ which is not parallel to $\vec{v}$.  Suppose:
\begin{enumerate}
\item  $\Gamma$ is symmetric about the reflection across a   plane $P$ parallel to the density vector $\vec{v}$,
\item   $P$ separates  $\Gamma$ in  two graphs on the straight-line $\Pi\cap P$. 
\end{enumerate}
If $\Sigma$ is an embedded  compact $\lambda$-translating soliton  with boundary $\Gamma$ and $\Sigma$ lies in one side of $\Pi$, then   $P$ is a symmetry plane of $\Sigma$.
\end{theorem}

\begin{proof} After a change of coordinates, we suppose $\vec{v}=(0,0,1)$.  Let us observe that the plane $\Pi$ is not necessarily horizontal (i.e. orthogonal to the direction $\vec{v}$).  After a rotation about the $z$-axis and up to a horizontal translation, we suppose that $P$ is the plane of equation $x=0$.  If $\Omega\subset\Pi$ is the domain bounded by $\Gamma$, the  embeddedness of $\Sigma$ ensures that  $\Sigma\cup\Omega$ defines a closed surface without boundary in $\r^3$, and possibly  non-smooth along $\Gamma$. In particular, $\Sigma\cup\Omega$ separates $\r^3$ in two connected components and we denote by    $W$ the bounded component. We orient $\Sigma$ by the Gauss map $N$ that points towards $W$.   

The proof is by contradiction and it is standard. By completeness, we give an outline on it. Suppose that $P$ is not a plane of symmetry of $\Sigma$, in particular, there exist two points $q_1, q_2\in \Sigma\setminus\Gamma$ such that the line $\overline{q_1q_2}$ joining $q_1$ to $q_2$ is orthogonal to $P$, $q_1$ and $q_2$ are in opposite sides of $P$ and $\mbox{dist}(q_1,P)>\mbox{dist}(q_2,P)$. Without loss of generality, suppose $x(q_2)<0<x(q_1)$. Then the symmetric point of $q_1$ about $P$, say $q_1^*$, satisfies $x(q_1^*)<x(q_2)$.  For any $t\in\r$, denote $P_t$ the plane of equation $x=t$. Let  $\Sigma(t)^{-}=\Sigma\cap \{x\leq t\}$, $\Sigma(t)^{+}=\Sigma\cap \{x\geq t\}$ and $\Sigma(t)^*$ the reflection of $\Sigma(t)^{+}$ about $P_t$. We notice that the reflection about $P_t$ preserves the value of $H_\phi$ as well as this reflection leaves invariant as a subset, the boundary plane $\Pi$. For $t$ sufficiently big and since $\Sigma$ is compact, we have $P_t\cap \Sigma=\emptyset$. Then we move $P_t$ towards $\Sigma$ by letting $t\searrow 0$, until the first contact point with $\Sigma$ at the time $t_1>0$. Then we move slightly more $P_{t_1}$, $t<t_1$, and we reflect $\Sigma(t)^{+}$.  The embeddness of $\Sigma$ and the fact that $\Sigma$ lies below $\Pi$ assures the existence of $\epsilon>0$ such that $\mbox{int}(\Sigma(t)^*)\subset W$ for every $t\in (t_1-\epsilon,t_1)$. By the compactness, there exists $t_2\geq 0$ with $t_2<t_1$, such that $\mbox{int}(\Sigma(t)^*)$ lies outside $W$ for any $t<t_2$. In fact, $t_2>0$ by the existence of the points $q_1$ and $q_2$ and $0<x(q_1^*)<x(q_2)$. Furthermore, and because $\Sigma(t)^{+}$ is a graph of $P_t$ for $t>t_2$ and $\Gamma\cap\{x>0\}$ is also a graph on the line $P\cap \Pi$, we have $\partial \Sigma(t_2)^*\cap \Gamma\subset P_{t_2}$. Then we have two possibilities:
\begin{enumerate}
\item There exists $p\in \mbox{ int}(\Sigma(t_2)^*)\cap  \mbox{ int} (\Sigma(t_2)^{-})$.  Since $p$ is an interior point, $\Sigma(t_2)^*$ and $\Sigma(t_2)^{-}$ are tangent at $p$ and the orientations of both surfaces agree at $p$ because they point towards $W$.  The tangency principle and the analyticity of $\Sigma$ imply that $\Sigma(t_2)^*=\Sigma(t_2)^{-}$, and thus,  $P_{t_2}$ is a plane of symmetry of $\Sigma$: a contradiction because $\Gamma$ is not invariant by reflections across $P_{t_2}$.
\item The surface $\Sigma$ is orthogonal to $P_{t_2}$ at some point $p\in\partial \Sigma(t_2)^{*}\cap \partial \Sigma(t_2)^{-}$. We now use the boundary version of the tangency principle and we conclude that $P_{t_2}$ is a plane of symmetry of $\Sigma$, a contradiction again.
\end{enumerate}
\end{proof}

The result may not be true in case that $P$ separates $\Gamma$ in two symmetric pieces that were not graphs on $\Pi\cap P$ because it has been crucial in the above proof to prevent the case that the contact point $p$ could belong to $(\partial \Sigma(t_2)^*\cap\Gamma)\setminus P_{t_2}$. 

A particular case of Th. \ref{t3} is when  $\Gamma$ is a circle contained in a  plane $\Pi$ orthogonal to $\vec{v}$.

\begin{corollary}\label{c2} Let $\Gamma$ be a circle contained in a plane $\Pi$ orthogonal to $\vec{v}$. The only embedded compact $\lambda$-translating soliton with boundary $\Gamma$ that  lies in one side of $\Pi$ is a rotational surface whose rotation axis is parallel to $\vec{v}$.
\end{corollary}

In case of translating solitons ($\lambda=0$) we can drop in Cor. \ref{c2} the hypothesis on the embeddedness and the fact that the surface lies in one side of $\Pi$.   For this, we need  the following result, which makes its own interest (when $\lambda=0$ and the boundary is convex, the technique was employed in \cite{pyo}).

\begin{lemma}\label{lem} Let $\Sigma$ be an embedded compact $\lambda$-translating soliton whose boundary is a Jordan  curve contained in a plane $\Pi$ which is not parallel to $\vec{v}$. Let $D\subset\Pi$ be the  domain bounded by $\partial\Sigma$ and let $\vec{a}$ be a unit vector orthogonal to $\Pi$. If $\mbox{int}(\Sigma)$ is contained in the  solid right cylinder $D\times\r\vec{a}$, then $\Sigma$ is a graph on $\Pi$.  In case $\lambda=0$, we can drop the hypothesis on the embeddedness of $\Sigma$.
\end{lemma}

\begin{proof} After a change of coordinates, we suppose $\vec{a}=(0,0,1)$ and $\Pi$ the plane of equation $z=0$. Let 
$$T=\Sigma\cup (D\times (-\infty,0]\vec{a}).$$
Then $T$ is an embedded surface (non smooth along $\partial\Sigma$) that separates $\r^3$ in two connected components. Let $W\subset\r^3$ be the component that contains the point $(q_1,q_2,q_3-1)$, where  $q=(q_1,q_2,q_3)\in\mbox{int}(\Sigma)$ is a point such that $z(q)=\min\{z(p):p\in\Sigma\}$. In particular, $D\times (-\infty,q_3-1]\vec{a}\subset W$. Let $N$ be the orientation on $\Sigma$ pointing towards $W$. 

By contradiction, we suppose that $\Sigma$ is not a graph. We take $\Sigma'$  a copy of $\Sigma$ and we move upwards $\Sigma'$ in the direction of $\vec{a}$  until that $\Sigma'\cap\Sigma=\emptyset$. We come back $\Sigma'$ until the first intersection point  $p$ with $\Sigma$. Since we are assuming that $\Sigma$ is not a graph on $\Pi$, then $p$ is a common interior point of $\Sigma'$ and $\Sigma$, and the touching between $\Sigma'$ and $\Sigma$ occurs before $\Sigma'$ coincides with $\Sigma$. Let us observe that the orientations of $\Sigma$ and $\Sigma'$ coincides at $p$ because both point towards $W$. Since $\Sigma'$ lies in one side of $\Sigma$, the tangency principle implies that $\Sigma'=\Sigma$, a contradiction because $\partial\Sigma'\not=\partial\Sigma$.

If $\lambda=0$ and $\Sigma$ is only immersed, the existence of $W$ does not make sense, but  the above comparison argument  between $\Sigma$ and $\Sigma'$ holds at the interior point $p$ because we can reverse the orientations, if necessary, and $\Sigma$ and $\Sigma'$ follow being translating solitons. \end{proof}

 Hence we recover the following result proved in \cite{pyo}.

 \begin{corollary}\label{c1}
  The only compact  translating soliton with circular boundary contained in a  plane orthogonal to $\vec{v}$ is a  rotational surface whose rotation axis is parallel to $\vec{v}$.
 \end{corollary}
 
 \begin{proof} Since the boundary is convex,  the interior of $\Sigma$ is included in the solid cylinder $\Omega\times\r\vec{v}$  by Cor. \ref{prv}. Lemma \ref{lem} yields   that $\Sigma$ is a graph,  in particular, $\Sigma$ is embedded. By Th. \ref{t2}   we know that $\Sigma$ lies in one side of  $\Pi$, and finally we apply Cor. \ref{c2}.
  \end{proof}
 
  From the above results, we have the next open questions:
  \begin{enumerate}
  \item[Q1] Is a compact translating soliton with circular boundary  a surface of revolution? Here the density vector $\vec{v}$ is arbitrary.   By Lem. \ref{lem} and Th. \ref{t3}, we know that the surface would be invariant by the reflection about a plane parallel to $\vec{v}$. However we think that if $\Pi$ is not orthogonal to $\vec{v}$, then there do not exist a translating soliton with circular boundary.
    \item[Q2] Is  an embedded compact $\lambda$-translating soliton with circular boundary a surface of revolution? In the case that the boundary plane is orthogonal to the density vector $\vec{v}$, we think that the answer is `yes'. 
    \end{enumerate}

 In Th. \ref{t3} we have prescribed   the boundary of the surface. Following \cite{pyo}, our second result replaces the symmetry of the boundary curve by the constancy of the angle between the surface and the  boundary plane. We extend the item (2) of the Main Theorem of \cite{pyo} as follows. 

 \begin{theorem}\label{t4}
Let  $\Sigma$ be an embedded compact  $\lambda$-translating soliton whose boundary is contained in a plane $\Pi$ and $\Pi$ is not parallel to the density vector $\vec{v}$.  If $\Sigma$ makes a constant contact angle with $\Pi$ along $\partial \Sigma$ and $\Sigma$ lies in one side of $\Pi$, then  $\Sigma$ has a symmetry about a  plane parallel to $\vec{v}$.
  \end{theorem}

  \begin{proof} The proof uses again the Alexandrov reflection method.   We use the same notation as in Th. \ref{t3} and we only point out  the differences. We suppose $\vec{v}=(0,0,1)$ again. Let $\vec{w}$ be a horizontal vector and parallel to $\Pi$: if $\Pi$ is not a horizontal plane, this  direction $\vec{w}$ is unique.  After a rotation about the $z$-axis, we suppose that $\vec{w}=(1,0,0)$.  Let $\{P_t\}_{t\in\r}$ be the foliation of $\r^3$ by planes of equation $x=t$.
  
  We begin with the reflection method as in Th. \ref{t3}. After the first time $t=t_1$ which  $P_{t_1}$ touches $\Sigma$, we arrive until $t=t_2$ where $\mbox{int}(\Sigma(t)^*)$ lies outside $W$ for every $t<t_2$. We have the next possibilities:
  \begin{enumerate}
  \item There exists $p\in \mbox{int}(\Sigma(t_2)^*)\cap \mbox{int}(\Sigma(t_2)^{-})$.
  \item $\Sigma$ is orthogonal to $P_{t_2}$ at some point $p\in \partial \Sigma(t_2)^*\cap \partial \Sigma(t_2)^{-}\setminus\Gamma$.
    \item There exists $p\in\partial \Sigma(t_2)^*\cap\partial \Sigma(t_2)^{-}\cap \Gamma$ and $p\not\in P_{t_2}$.
   \item $\Sigma$ is orthogonal to $P_{t_2}$ at some point $p\in \partial \Sigma(t_2)^*\cap \partial \Sigma(t_2)^{-}\cap\Gamma$.

   \end{enumerate}

The cases (1) and (2) appeared in Th. \ref{t3} and the tangency principle  implies  that $P_{t_2}$ is a plane of symmetry of $\Sigma$.  In case (3), we have $T_p\partial \Sigma(t_2)^*=T_p\partial \Sigma(t_2)^{-}$ and the surfaces $\Sigma(t_2)^*$ and $\Sigma(t_2)^{-}$ are tangent at $p$ because the contact angle between $\Sigma(t_2)^*$ and $P_{t_2}$ agrees with the one between $\Sigma(t_2)^{-}$ and $P_{t_2}$: here we use that $P_{t_2}$ is a vertical plane and that $\Pi$ is invariant by reflection about $P_{t_2}$. Then we can apply the boundary version of the tangency principle to prove that   $P_{t_2}$ is a plane of symmetry of $\Sigma$. In case (4), the tangency principle at a corner point (\cite{se2}) proves that $\Sigma(t_2)^{-}=\Sigma(t_2)^*$ and thus, $P_{t_2}$ is a plane of symmetry of $\Sigma$ again.
   \end{proof}
   
As in Cor. \ref{c1},  a consequence of Ths. \ref{t2} and \ref{t4} is:

  \begin{corollary}\label{c-c2} Let  $\Sigma$ be an embedded compact  translating soliton  whose boundary is contained in a   plane $\Pi$ and $\Pi$ is orthogonal to the density vector $\vec{v}$.    If $\Sigma$ makes a constant contact angle with $\Pi$ along $\partial \Sigma$, then  $\Sigma$ is a surface of revolution whose rotation axis is parallel to $\vec{v}$ and $\partial \Sigma$ is a circle.
 \end{corollary}

If we compare Cor. \ref{c-c2}  with statement (2) in the Main Theorem of \cite{pyo},  we  have assumed that $\Sigma$ is embedded whereas in  \cite{pyo} the embeddedness is replaced by the convexity of $\partial\Sigma$:  Corollary \ref{prv} and Lem. \ref{lem}   concludes that if $\partial\Sigma$ is convex, then $\Sigma$ is a graph and so, it is embedded.

As we have observed, in order to apply Th. \ref{t3}, we need to assure that the surface lies in one side of the boundary plane. Our last result proves that it suffices to assume that the surface is transverse to the boundary plane along its boundary. The next result extends  a similar case for cmc surfaces: see \cite{bemr}.
 
\begin{theorem}\label{tfi} Let $\Pi$ be a plane orthogonal to the density vector $\vec{v}$. Let $\Sigma$ be an embedded compact $\lambda$-translating soliton such that its boundary $\partial\Sigma$ is a convex curve contained in $\Pi$.  If in a neighbourhood of $\partial\Sigma$, $\Sigma$ lies in the halfspace of $\r^3\setminus \Pi$ where points $\vec{v}$ and $\Sigma$ is    transverse to $\Pi$ along $\partial\Sigma$, then $\Sigma$ lies in one side of $\Pi$ and thus, $\Sigma$ inherits all the symmetries of $\partial\Sigma$.
\end{theorem}

\begin{proof} The case $\lambda=0$ was proved in Th. \ref{t2}. Suppose $\lambda\not=0$. After a change of coordinates, we assume $\vec{v}=e_3=(0,0,1)$ and $\Pi$ is the plane of equation $z=0$. The hypothesis says that $\Sigma$ is contained in the halfspace $z>0$ in a neighbourhood of  $\Gamma=\partial\Sigma$ and $\langle\nu,e_3\rangle>0$ along $\Gamma$, where   $\nu$  is the unit inward conormal vector of $\Sigma$ along $\Gamma$. We will prove that $\mbox{int}(\Sigma)\subset \{z> 0\}$. By contradiction, we suppose that $\Sigma\cap\Pi$ has other components than $\Gamma$.

 Let $D\subset\Pi$ be the domain bounded by $\Gamma$ and denote $\mbox{ext}(D)=\Pi\setminus\overline{D}$. We know by Th. \ref{tko} that it is not possible  that all components of $\Sigma\cap\Pi$ other than $\Gamma$ are in $D$. Once proved this, and by using the transversality of $\Sigma$ along $\Gamma$, we now   construct a suitable embedded closed   surface $\tilde{\Sigma}$ by removing from $\Sigma$ some annuli that across $D$,  attaching some horizontal disks  and finally the very domain $D$: we refer to the reader to  Th. 1 in \cite{bemr} for details. Then $\tilde{\Sigma}$ separates $\r^3$ into two connected components and denote by $W$ the bounded component. Consider on $\tilde{\Sigma}$ the orientation $N$ pointing towards $W$.  By applying the Alexandrov reflection method by means of reflections about vertical planes, it is easy to prove that it is not possible that there exist components of $\tilde{\Sigma}\cap \mbox{ext}( D)$ nullhomologous in $\mbox{ext}(D)$, neither, two or more components in $\tilde{\Sigma}\cap \mbox{ext}(D)$: in the first case,  we use the convexity of $\Gamma$. 

Finally, the last case to consider is that $\tilde{\Sigma}\cap \mbox{ext}(D)$  has exactly one component $C$.  Since $N$ points towards $W$, then   along $\Gamma\cup C$, the vector $N$ points into the annulus in $\mbox{ext}(D)$ bounded by $\Gamma\cup C$. In particular, $N$ points towards $\mbox{ext}(D)$ along $\Gamma$ and this implies that the orientation $\eta_D$ induced by $\tilde{\Sigma}$ in $D$ is  $\eta_D=-e_3$. We prove that with the orientation $N$, the value of $\lambda$ is positive. Fix $P$ a  plane parallel to $\vec{v}$ sufficiently  far so $P$ does not intersect $\tilde{\Sigma}$. Then we move $P$ parallel towards $\tilde{\Sigma}$ until the first touching point. By the existence of the component $C$, the first touching point between $P$ and $\tilde{\Sigma}$ occurs at some interior point $p$ of $\tilde{\Sigma}$. Since $P$ is a translating soliton, and $N$ points to $W$, then $\tilde{\Sigma}$ lies above $P$ around $p$ and  the tangency principle implies   $\lambda>0$. 

Finally, by integrating (\ref{lapla}) in $\Sigma$, we obtain
\begin{equation}\label{nnn}-\int_{\Gamma}\langle\nu,e_3\rangle ds=2\lambda\int_{ \Sigma}\langle N,e_3\rangle\ d\Sigma+\int_{\Sigma}\langle N,e_3\rangle^2\ d\Sigma.
\end{equation}
Since $\eta_D=-e_3$, the divergence theorem implies  $\int_\Sigma\langle N,e_3\rangle =-\int_D\langle \eta_D,e_3\rangle=\mbox{area}(D)$. Thus (\ref{nnn}) is now
$$-\int_{\Gamma}\langle\nu,e_3\rangle ds=2\lambda\mbox{ area}(D)+\int_{\Sigma}\langle N,e_3\rangle^2\ d\Sigma.$$
However the left-hand side is negative and the right-hand side is positive, obtaining a contradiction. 
\end{proof}


\end{document}